\documentclass[a4paper,11pt]{article}
\usepackage{amssymb,amsmath,amsthm,amsfonts}
\usepackage{mathrsfs}
\usepackage{multirow}
\usepackage{graphicx}
\usepackage{authblk}
\usepackage{indentfirst}
\usepackage{multicol}
\usepackage{url}
\usepackage{fancyhdr}
\usepackage[all]{xy}
\usepackage{mathtools}
\usepackage[english]{babel}
\usepackage{colortbl}
\usepackage{caption}
\usepackage{subcaption}
\usepackage{tikz}
\usepackage{float}
\usepackage{enumitem}
\usepackage{listings}
\usepackage{hyperref} 
\usepackage{makecell}
\usepackage{cite}
\usepackage{tikz-cd}
\usepackage[all]{xy}
\usepackage{graphicx}
\usepackage{amssymb}
\usepackage{array}
\usepackage{mathrsfs}
\usepackage{booktabs}
\usepackage{extarrows}
\usepackage{verbatim}
\usepackage{tikz}
\usepackage{amsmath}
\usepackage{subcaption}

\newtheorem{theorem}{Theorem}[section]
\newtheorem{proposition}[theorem]{Proposition}
\newtheorem{lemma}[theorem]{Lemma}

\newtheorem{definition}{Definition}[section]
\newtheorem{example}{Example}[section]

\newcommand{\aut}{{\mathrm{Aut}\hspace{0.1em}}}

\usepackage{xr}
\makeatletter
    \newcommand*{\addFileDependency}[1]{
    \typeout{(#1)}
    \@addtofilelist{#1}
    \IfFileExists{#1}{}{\typeout{No file #1.}}
    }
\makeatother

\newtheorem{Theorem}{Theorem}[section]
\newtheorem{Lemma}[Theorem]{Lemma}
\theoremstyle{Definition}
\newtheorem{Definition}[Theorem]{Definition}
\newtheorem{Example}[Theorem]{Example}

\theoremstyle{Remark}

\theoremstyle{notation}

\title{Multi-scale symmetry analysis in molecular structures}

\author[1,2]{Jing-Wen Gao}
\author[1]{Yunan He}
\author[1]{Jian Liu\thanks{Corresponding author: liujian@cqut.edu.cn}}

\affil[1]{Mathematical Science Research Center, Chongqing University of Technology, Chongqing 400054, China}
\affil[2]{Mathematical Science School, Huazhong University of Science and Technology, Wuhan 430074, China}

\makeatletter
    \renewcommand*{\@fnsymbol}[1]{\ensuremath{\ifcase#1\or \dagger\or *\or *\or
   \mathsection\or \else\@ctrerr\fi}}
\makeatother
\date{}

\begin{document}
    \maketitle

    \paragraph{Abstract}
    Topological data analysis (TDA), as a relatively recent approach, has demonstrated great potential in capturing the intrinsic and robust structural features of complex data. While persistent homology, as a core tool of TDA, focuses on characterizing geometric shapes and topological structures, the automorphism groups of Vietoris-Rips complexes can capture the structured symmetry features of data. In this work, we propose a multi-scale symmetry analysis approach that leverages persistent automorphism modules to quantify variations in symmetries across scales. By modifying the category of graphs and constructing a suitable functor from the graph category to the category of modules, we ensure that the persistent automorphism module forms a genuine persistence module. Furthermore, we apply this framework to the structural analysis of fullerenes, predicting the stability of 12 fullerene molecules with a competitive correlation coefficient of 0.979.

    \paragraph{Keywords}
     Symmetry, automorphism group, persistent automorphism, persistence module, fullerene molecule.

\footnotetext[1]
{ {\bf 2020 Mathematics Subject Classification.}  	Primary  55N31; Secondary 20B25, 52B70.
}

\tableofcontents 

\section{Introduction}

Persistent homology, developed over the past two decades, has become a fundamental method for capturing the geometric shape and topological structure of data. It has achieved significant success, particularly in the characterization of structured and robust features in complex datasets \cite{GC2005,GC2009,HE2002}. Over time, various extensions of persistent homology, such as persistent cohomology \cite{de2011persistent, cang2020persistent, schonsheck2024spherical} and persistent Laplacians \cite{memoli2022persistent, liu2023algebraic, wang2020persistent}, have also been developed.

The core idea of persistent homology is to capture the changes in topological structures of data, such as connected components, loops, and cavities, at different scales. These topological structures represent intrinsic, stable invariants in both space and data. In contrast, geometric information, such as symmetry, is also a critical feature. The automorphism group of a space provides a powerful algebraic tool to describe its symmetries, making it an essential part of understanding the underlying structure and behavior of spatial objects. Motivated by persistent homology, we aim to combine the multi-scale information from persistent homology with symmetry features for data analysis.

Currently, the application of symmetry in data analysis is primarily focused on symmetry detection in images or geometric shapes \cite{marola2002detection,sun2011reflection,mitra2006partial}. Recently, the work in \cite{liu2025parametrization} develops a comprehensive theory for analyzing the persistent symmetries and degrees of asymmetry in finite point configurations within metric spaces. However, the use of symmetry in data to analyze the structure and distribution of the data itself is still in its early stages.

In this work, we propose a multi-scale symmetry analysis (MSA) method for analyzing point cloud data. First, for point cloud data, we construct the corresponding Vietoris-Rips complex, whose automorphism group is isomorphic to the proximity graph of the point cloud. This allows us to transform the study of the automorphism group of simplicial complexes into the study of the automorphism group of graphs. Furthermore, we improve the graph category and construct a functor from the graph category to the module category, leading to the construction of persistent automorphism modules. In the application, we consider the order and symmetry degree of the automorphism group as features that characterize the richness of symmetry, which are then applied to data analysis. Finally, we analyze the symmetry of fullerenes to examine the stability of their structure, achieving a relatively high correlation coefficient of 0.979.

The paper is organized as follows. In the next section, we review the automorphisms of graphs and simplicial complexes. Section \ref{section:persistence} introduces the main methods. In Section \ref{section:application}, multi-scale symmetry analysis is applied to the stability analysis of fullerenes. The final section provides a summary and conclusion.

\section{Combinatorial automorphism group}

The automorphism group is a crucial geometric feature, reflecting the symmetry of a geometric object. Our work aims to study the symmetry of data, which precisely requires the automorphism groups of graphs and simplicial complexes as a foundation. In this section, we recall some fundamental concepts related to automorphism groups of graphs and simplicial complexes. A more detailed discussion on the automorphism groups of graphs or simplicial complexes can be found in the works of \cite{armstrong1988presentation, LB1995, DB2001, de2000topics}.

\subsection{Automorphism group of graphs}

\begin{definition}
{\rm An \emph{automorphism} of a graph $G = (V, E)$ is a permutation $\sigma: V \to V$ of the vertex set $V$ such that the edge set $E$ is preserved, i.e., $\{u, v\} \in E$ if and only if $\{\sigma(u), \sigma(v)\} \in E$. The set of all automorphisms of $G$, denoted $\operatorname{Aut}(G)$, forms a group under function composition, called the \emph{automorphism group} of $G$, defined as
\[
\operatorname{Aut}(G) = \{\sigma: V \to V \mid \sigma \text{ is a bijection and } \{u, v\} \in E \iff \{\sigma(u), \sigma(v)\} \in E\}.
\]}
\end{definition}

In other words, an automorphism $\sigma \in \operatorname{Aut}(G)$ is a bijection of the vertex set $V$ that preserves adjacency: for all $u,v \in V$, $\{u,v\} \in E$ if and only if $\{\sigma(u),\sigma(v)\} \in E$.

The order of the automorphism group $\operatorname{Aut}(G)$ reflects the symmetry of a graph $G$. A larger order indicates that more vertex permutations preserve the edge structure, meaning the graph is highly symmetric. Conversely, if $|\operatorname{Aut}(G)| = 1$, the graph is asymmetric, having no nontrivial automorphisms. For example, a complete graph or an empty graph on $n$ vertices has $|\operatorname{Aut}(G)| = n!$, indicating maximal symmetry, while most random graphs have a trivial automorphism group, reflecting minimal symmetry. Thus, the order of $\operatorname{Aut}(G)$ is a simple quantitative measure of a graph's symmetry.

In Table~\ref{table:examples}, we present the automorphism groups of various graphs. It is evident that even graphs with very different structures, such as the complete graph $K_5$ and the Petersen graph, can possess automorphism groups of the same order. Nevertheless, this does not prevent us from regarding both graphs as highly symmetric.

\begin{table}[h]
\centering
\begin{tabular}{c|c|c|c|c}
\hline
Graph & Vertices & $\operatorname{Aut}(G)$ & Order & Symmetry Description \\
\hline
Complete graph $K_n$ & $n$ & $S_n$ & $n!$ & Highly symmetric \\
Empty graph $\overline{K_n}$ & $n$ & $S_n$ & $n!$ & Highly symmetric \\
Path $P_n$ & $n$ & $\mathbb{Z}_2$ & $2$ & Symmetry limited (flip) \\
Cycle $C_n$ & $n$ & $D_{2n}$ & $2n$ & Rotations + reflections \\
Petersen graph & 10 & $S_5$ & 120 & Highly symmetric \\
Complete bipartite graph $K_{m,n}$ & $m+n$ & $S_m \times S_n$ & $m! n!$ & Partite sets permutable \\
\hline
\end{tabular}
\caption{Examples of graph automorphism groups}\label{table:examples}
\end{table}

\subsection{Automorphism group of simplcial complexes}

\begin{definition}
{\rm Let $K$ be an abstract simplicial complex with vertex set $V(K)$. An \textbf{automorphism} of $K$ is a permutation $\varphi : V(K) \to V(K)$ such that for every simplex $\sigma \subseteq V(K)$, if $\sigma \in K$ then $\varphi(\sigma) \in K$. In other words, $\varphi$ preserves the simplicial structure of $K$.

The set of all automorphisms of $K$, equipped with composition, forms a group called the \textbf{automorphism group} of $K$, denoted by $\mathrm{Aut}(K)$.}
\end{definition}

\begin{definition}
{\rm Let $G$ be a finite simple graph with vertex set $V(G)$. The \emph{flag complex} (also called the \emph{clique complex}) of $G$, denoted by $F(G)$, is the simplicial complex whose vertex set is $V(G)$, and where a finite set of vertices $\{v_{0}, \dots, v_{k}\} \subseteq V(G)$ spans a $k$-simplex in $F(G)$ if and only if these vertices form a complete subgraph (clique) in $G$.}
\end{definition}

Equivalently, $F(G)$ is the maximal simplicial complex having $G$ as its $1$-skeleton.

\begin{theorem}\label{theorem:isomorphism}
Let $G$ be a finite simple graph, and let $F(G)$ denote its clique complex (i.e., the flag complex generated by the cliques of $G$). Then there is a natural group isomorphism
\[
\operatorname{Aut}(G) \;\cong\; \operatorname{Aut}(F(G)).
\]
\end{theorem}

\begin{proof}
Define a map
\[
\Phi : \operatorname{Aut}(G) \longrightarrow \operatorname{Aut}(F(G))
\]
as follows: given a graph automorphism $\sigma \in \operatorname{Aut}(G)$, let $\Phi(\sigma)$ act on the vertices of $F(G)$ in the same way, and extend linearly to all simplices. Since $\sigma$ preserves adjacency, it sends every clique to a clique, hence $\Phi(\sigma)$ is a simplicial automorphism. Thus $\Phi$ is a well-defined group homomorphism.

We first show that $\Phi$ is injective. If $\Phi(\sigma)$ is the identity on $F(G)$, then it fixes every vertex of $F(G)$. Since the vertex set of $F(G)$ is exactly the vertex set of $G$, it follows that $\sigma$ is the identity automorphism of $G$. Thus $\ker(\Phi) = \{ e \}$.

To prove surjectivity, let $\tau \in \operatorname{Aut}(F(G))$ be a simplicial automorphism. Then $\tau$ induces a bijection of the vertex set preserving all simplices of $F(G)$. Since $F(G)$ is a flag complex, it is uniquely determined by its $1$-skeleton, which is exactly the graph $G$. Therefore $\tau$ preserves edges of $G$, and hence corresponds to a graph automorphism $\sigma \in \operatorname{Aut}(G)$. By construction, $\Phi(\sigma) = \tau$. This shows that $\Phi$ is surjective.

Therefore $\Phi$ is an isomorphism, and we obtain the natural group isomorphism
\[
\operatorname{Aut}(G) \;\cong\; \operatorname{Aut}(F(G)).
\]
This completes the proof.
\end{proof}

\begin{example}
{\rm Let $G$ be the graph obtained from the triangle $C_3$ with vertices $\{0,1,2\}$ and edges $\{\{0,1\},\{1,2\},\{2,0\}\}$ by adding a new vertex $3$ and an edge $\{0,3\}$. The vertex set is $V(G) = \{0,1,2,3\}$ and the edge set is
\[
E(G) = \{\{0,1\}, \{1,2\}, \{2,0\}, \{0,3\}\}.
\]
The flag complex $F(G)$ has vertices $\{0,1,2,3\}$, edges corresponding to the edges of $G$, and a $2$-simplex corresponding to the triangle $\{0,1,2\}$. There is no higher simplex because the new vertex $3$ is only connected to $0$.

The automorphism group of $G$, $\operatorname{Aut}(G)$, consists of the permutations of vertices that preserve adjacency. Here, $0$ is distinguished as the unique vertex of degree $3$, so it must be fixed, while $1$ and $2$ can be swapped. Vertex $3$ is connected only to $0$, so it must also be fixed. Therefore,
\[
\operatorname{Aut}(G) \;\cong\; \mathbb{Z}_2.
\]
Similarly, any simplicial automorphism of $F(G)$ must fix $0$ and $3$ and may swap $1$ and $2$, so we also have
\[
\operatorname{Aut}(F(G)) \;\cong\; \mathbb{Z}_2.
\]
Thus, this example illustrates that $\operatorname{Aut}(G) \;\cong\; \operatorname{Aut}(F(G))$.}
\end{example}

Let $X$ be a set of points in a metric space $(M,d)$. The \emph{Vietoris-Rips complex} at scale $\varepsilon>0$, denoted by $\mathcal{R}_\varepsilon(X)$, is the simplicial complex with vertex set $X$ such that a finite subset $\sigma \subseteq X$ spans a simplex if and only if
\[
d(x,y) \le \varepsilon \quad \text{for all } x,y \in \sigma.
\]
We can also define a graph associated with $X$ at scale $\varepsilon>0$, called the \emph{proximity graph} $G_\varepsilon(X)$, as follows. The vertex set is $X$, and two distinct vertices $x,y \in X$ are connected by an edge if and only if
\[
d(x,y) \le \varepsilon.
\]

Equivalently, $G_\varepsilon(X)$ is the $1$-skeleton of the Vietoris-Rips complex $\mathcal{R}_\varepsilon(X)$.

\begin{proposition}\label{proposition:equation}
Let $X$ be a set of points in a metric space $(M,d)$ and $\varepsilon>0$.  Then we have
\[
\mathcal{R}_\varepsilon(X) = F(G_\varepsilon(X)).
\]
\end{proposition}

\begin{proof}
Recall that the flag complex $F(G_\varepsilon(X))$ is the simplicial complex whose simplices are exactly the cliques in $G_\varepsilon(X)$. By definition of $G_\varepsilon(X)$, a set of vertices $\sigma \subseteq X$ forms a clique if and only if $d(x,y)\le \varepsilon$ for all $x,y \in \sigma$.

On the other hand, $\sigma$ spans a simplex in the Vietoris-Rips complex $\mathcal{R}_\varepsilon(X)$ if and only if $d(x,y)\le \varepsilon$ for all $x,y \in \sigma$. Therefore, the simplices of $\mathcal{R}_\varepsilon(X)$ are precisely the cliques of $G_\varepsilon(X)$, which proves that $\mathcal{R}_\varepsilon(X) = F(G_\varepsilon(X))$.
\end{proof}

\begin{proposition}
Let $X$ be a set of points in a metric space $(M,d)$ and $\varepsilon>0$. Then we have
\[
\operatorname{Aut}(G_\varepsilon(X)) \;\cong\; \operatorname{Aut}(\mathcal{R}_\varepsilon(X)).
\]
\end{proposition}
\begin{proof}
This result follows directly from Theorem \ref{theorem:isomorphism} and Proposition \ref{proposition:equation}.
\end{proof}

\section{Persistent automorphisms}\label{section:persistence}

Introducing multi-scale information provides an effective approach to more accurately capturing the symmetry of data. This naturally leads to the study of persistent automorphism groups of spaces. Given a point cloud, we can construct the Vietoris-Rips complex. As demonstrated in the previous section, the automorphism group of this Rips complex is isomorphic to the automorphism group of the proximity graph of the point cloud. This equivalence allows us to reformulate the problem of determining the automorphism group of a space in terms of the automorphism group of a graph. This section is devoted to the study of automorphism groups of graphs.

\subsection{Functorial construction of automorphism groups}

We aim to construct a functor
\[
  \operatorname{Aut}: \mathbf{Graph} \to \mathbf{Group}
\]
from the category of graphs to the category of groups. However, this is not a natural construction, since graph homomorphisms do not necessarily induce homomorphisms between automorphism groups in a natural way. When a graph homomorphism \( f: H \to G \) is an isomorphism, we obtain a map
\[
  \tilde{f}: \operatorname{Aut}(H) \to \operatorname{Aut}(G), \quad \phi \mapsto f \circ \phi \circ f^{-1}.
\]
However, when \( f: H \to G \) is merely a graph homomorphism, it is challenging to provide a corresponding morphism between the automorphism groups.

\begin{example}
{\rm Let \( G = P_{3} \) be the path graph with vertex set \( \{1, 2, 3\} \) and edges \( \{1, 2\} \) and \( \{2, 3\} \). Therefore, we can obtain the automorphism group
\[
\operatorname{Aut}(G) \cong \mathbb{Z}/2,
\]
which is generated by the reflection swapping \( 1 \leftrightarrow 3 \) and \( 2 \leftrightarrow 2 \).

Let \( H \) be the subgraph of \( G \) induced by the vertex set \( \{1, 2\} \). Hence, \( H \) is the single edge \( \{1, 2\} \), and
\[
\operatorname{Aut}(H) \cong \mathbb{Z}/2.
\]
Consider the inclusion map
\[
i : H \hookrightarrow G.
\]

If there were a group homomorphism
\[
\operatorname{Aut}(i) : \operatorname{Aut}(H) \longrightarrow \operatorname{Aut}(G),
\]
then the non-trivial generator of \( \operatorname{Aut}(H) \), which swaps 1 and 2, could only be mapped to the trivial element in \( \operatorname{Aut}(G) \). This is because the automorphism of \( H \) that swaps 1 and 2 cannot be extended to a non-trivial automorphism of \( G \) via the inclusion map \( i \), as no non-trivial automorphism of \( G \) exists that fixes 2 while swapping 1 and 3.}
\end{example}

From now on, \( \mathbb{F} \) is assumed to be a field. We introduce the following subcategory \( \mathbf{Graph}^{\ast} \) of the category of graphs. The objects of \( \mathbf{Graph}^{\ast} \), denoted \( \text{obj}(\mathbf{Graph}^{\ast}) \), are all graphs. The morphisms are given by inclusions of graphs. For each graph \( G \), the identity map \( \text{id}_G \) is the identity morphism in \( \text{Hom}(G, G) \).

Let \( \mathbf{Vec}_\mathbb{F} \) denote the category of vector spaces over \( \mathbb{F} \). We define a contravariant functor
\[
\Phi: \mathbf{Graph}^{\ast} \to \mathbf{Vec}_\mathbb{F}.
\]
For each graph \( G \), we set
\[
\Phi(G) = \mathbb{F} \text{Aut}(G),
\]
where \( \mathbb{F} \text{Aut}(G) \) denotes the \( \mathbb{F} \)-linear space generated by the set \( \text{Aut}(G) \), the automorphism group of \( G \). If \( f: G_1 \hookrightarrow G_2 \), then \( \Phi(f): \mathbb{F} \text{Aut}(G_2) \to \mathbb{F} \text{Aut}(G_1) \) is given by
\[
\Phi(f)(\eta) =
\begin{cases}
\text{the restriction } \eta|_{G_1}, & \text{if } \eta|_G \in \text{Aut}(G) \text{ for any } G_1 \subseteq G \subseteq G_2, \\
0, & \text{otherwise}.
\end{cases}
\]

\begin{example}
{\rm Let $G_1$ be the 2-vertex graph with edge $v_1v_2$, and $G_2$ be the 3-vertex cycle with vertices $v_1,v_2,v_3$ and edges $v_1v_2,v_2v_3,v_3v_1$. Consider the embedding $f: G_1 \hookrightarrow G_2$
given by
\[
f(v_1)=v_1, \quad f(v_2)=v_2.
\]
The map
\[
\Phi(f): \mathbb{F}\aut(G_2) \to \mathbb{F}\aut(G_1)
\]
sends each $\eta \in \aut(G_2)$ to its restriction $\eta|_{G_1}$ if it is a valid automorphism of $G_1$, and to $0$ otherwise.

The automorphism groups are
\[
\aut(G_1) = \{\mathrm{id}, (v_1v_2)\}, \qquad
\aut(G_2) = \{ \mathrm{id}, (v_1v_2), (v_2v_3), (v_1v_3), (v_1v_2v_3), (v_1v_3v_2) \}.
\]
Then the map $\Phi(f)$ acts on the elements of $\aut(G_2)$ in Table \ref{table:elements}.
\begin{table}[htbp]
  \centering
\begin{tabular}{c|c|c}
$\eta \in \aut(G_2)$ & $\eta|_{G_1}$ & $\Phi(f)(\eta)$ \\
\hline
id & $v_1 \mapsto v_1, v_2 \mapsto v_2$ & id \\
$(v_1v_2)$ & $v_1 \mapsto v_2, v_2 \mapsto v_1$ & $(v_1v_2)$ \\
$(v_2v_3)$ & $v_1 \mapsto v_1, v_2 \mapsto v_3$ & 0 \\
$(v_1v_3)$ & $v_1 \mapsto v_3, v_2 \mapsto v_2$ & 0 \\
$(v_1v_2v_3)$ & $v_1 \mapsto v_2, v_2 \mapsto v_3$ & 0 \\
$(v_1v_3v_2)$ & $v_1 \mapsto v_3, v_2 \mapsto v_1$ & 0 \\
\end{tabular}
  \caption{Action of the restriction map $\Phi(f)$ on the elements of $\text{Aut}(G_2)$.}\label{table:elements}
\end{table}
Hence, the map $\Phi(f)$ restricts $S_3 = \aut(G_2)$ to its subgroup $\aut(G_1) \cong \mathbb{Z}/2$, sending all other elements to $0$.}
\end{example}

\begin{lemma}\label{composition}
Let \( f: G_1 \to G_2 \) and \( g: G_2 \to G_3 \) be morphisms in the category \( \mathbf{Graph}^{\ast} \). Then, we have
\[
\Phi(g \circ f) = \Phi(f) \circ \Phi(g).
\]
\end{lemma}

\begin{proof}
Let \( \eta \in \aut(G_3) \). By the definition of the functor \( \Phi \), we have
\[
\Phi(f)(\eta) =
\begin{cases}
\text{the restriction } \eta|_{G_1}, & \text{if } \eta|_G \in \aut(G) \text{ for all } G_1 \subseteq G \subseteq G_2, \\
0, & \text{otherwise}.
\end{cases}
\]
Next, apply \( \Phi(g) \) to \( \Phi(f)(\eta) \). By the definition of \( \Phi(g) \), we get
\[
\Phi(g)(\Phi(f)(\eta)) =
\begin{cases}
\text{the restriction } (\eta|_{G_1})|_{G_2}, & \text{if } (\eta|_{G_1})|_G \in \aut(G) \text{ for all } G_2 \subseteq G \subseteq G_3, \\
0, & \text{otherwise}.
\end{cases}
\]
Since \( \eta \) satisfies the condition for \( \Phi(f) \), it follows that \( \eta|_G \in \aut(G) \) for all \( G_3 \subseteq G \subseteq G_1 \). Therefore, we conclude that
\[
\Phi(g)(\Phi(f)(\eta)) =
\begin{cases}
\text{the restriction } \eta|_{G_2}, & \text{if } \eta|_G \in \aut(G) \text{ for all } G_2 \subseteq G \subseteq G_3, \\
0, & \text{otherwise}.
\end{cases}
\]
This is exactly the definition of \( \Phi(g \circ f) \). Thus, we have \( \Phi(g \circ f) = \Phi(f) \circ \Phi(g) \), which completes the proof of the lemma.
\end{proof}

\begin{proposition}
The construction \( \Phi: \mathbf{Graph}^{\ast} \to \mathbf{Vec}_{\mathbb{F}} \) is a contravariant functor.
\end{proposition}

\begin{proof}
By Lemma \ref{composition}, we have established that \( \Phi \) preserves the composition of morphisms in the opposite direction.

Next, we show that \( \Phi \) preserves identity morphisms. This follows directly from the definition of \( \Phi \). Indeed, for any \( \eta \in \aut(G) \), the condition \( \eta|_H = \eta \in \aut(H) = \aut(G) \) for any \( G \subseteq H \subseteq G \) automatically holds. Therefore, we have
\[
\Phi(\mathrm{id}_G): \mathbb{F} \aut(G) \to \mathbb{F} \aut(G),
\]
and for every \( \eta \in \aut(G) \), we have \( \Phi(\mathrm{id}_G)(\eta) = \eta \). This shows that
\[
\Phi(\mathrm{id}_G) = \mathrm{id}_{\mathbb{F} \aut(G)}.
\]
Thus, \( \Phi \) preserves identity morphisms. Since \( \Phi \) preserves both compositions and identities, we conclude that \( \Phi \) is a contravariant functor.
\end{proof}

\subsection{Persistent automorphism module}

\begin{Definition}
{\rm A \emph{persistence graph} \( \mathcal{G} = \{ G_i \}_{i \geq 0} \) in the category \( \mathbf{Graph}^{\ast} \) is a family of graphs, \( \{ G_i \}_{i \geq 0} \), together with graph maps \( f_i : G_i \to G_{i+1} \) in the category \( \mathbf{Graph}^{\ast} \), for each \( i \geq 0 \).}
\end{Definition}

Given a persistence graph \( \mathcal{G} = \{ G_i \}_{i \geq 0} \) in \( \mathbf{Graph}^{\ast} \), we have the following diagram, which represents a filtration of graphs
\[
  G_0 \xhookrightarrow{f_0} G_1 \xhookrightarrow{f_1} G_2 \xhookrightarrow{f_2} \cdots.
\]
\begin{Example}
{\rm Let \( G = (V, E) \) be a weighted graph with a weight function \( w: E \to \mathbb{R} \) that assigns a real number to each edge. For any real number \( a \), we define a subgraph \( G_a = (V, E_a) \), where the edge set \( E_a \) is given by:
\[
  E_a = \{ e \in E \mid w(e) \leq a \}.
\]
Thus, for a sequence of real numbers \( a_1\leq a_2\leq \cdots\leq a_k\leq\cdots \), we obtain a family of graphs \( \{ G_{a_i} \}_{i \geq 0} \), which forms a persistence graph. Specifically, we get the following filtration of graphs
\[
  G_{a_1} \xhookrightarrow{f_1} G_{a_2} \xhookrightarrow{f_2} G_{a_3} \xhookrightarrow{f_3} \cdots
\]
where \( f_i: G_{a_i} \to G_{a_{i+1}} \) is the natural inclusion map, given by
\[
  f_i: G_{a_i} \hookrightarrow G_{a_{i+1}}, \quad \text{for} \quad a_i \leq a_{i+1}.
\]

This construction represents a filtration of graphs, where the graph \( G_{a_i} \) evolves as the threshold \( a \) increases. Intuitively, as the parameter \( a \) increases, more edges are included in the graph, reflecting a growing connectivity structure.}

\end{Example}

\begin{Definition}
{\rm Given a persistence graph $\mathcal{G}=\{G_i, f_i\}$, for $p\geq 0$, the \emph{$p$-persistent automorphism module of $G_i$}, denoted $$\mathbb{F}{\rm Aut}(G_i)^{i, i+p},$$ is the image of the induced homomorphism $$\Phi(f_{i+p})\circ\cdots\circ\Phi(f_i): \mathbb{F}{\rm Aut}(G_{i+p})\rightarrow \mathbb{F}{\rm Aut}(G_{i}).$$}
\end{Definition}

\begin{Definition}
{\rm  A \emph{persistence module} $\mathcal{M}$ is a family of linear spaces $\{M_i\}_{i\geq 0}$, together with linear maps $\varphi_i: M_i\rightarrow M_{i+1}$.

Dually, a family of linear spaces $\{M_i\}_{i\geq 0}$, together with linear maps $\varphi_i: M_{i+1}\rightarrow M_{i}$ is a \emph{(co)persistence module}.}
\end{Definition}

For example, if $\mathcal{G}=\{G_i, f_i\}$ is a persistence graph, then
 \begin{equation}\label{persistence module}
\{\mathbb{F}\aut(G_i), \Phi(f_i)\}
\end{equation} is a persistence module.

\begin{Definition}
{\rm A persistence graph $\mathcal{G}=\{G_i, f_i\}$ is of \emph{finite type} if each component graph $G_i$ is a finite graph, and if the graph maps $f_i$ are graph isomorphisms for $i\geq m$ for some integer $m$.
	
A persistence module  $\mathcal{M}=\{M_i, \varphi_i\}$ is of \emph{finite type} if each component $\mathbb{F}$-linear space $M_i$ is a finitely generated $\mathbb{F}$-linear space, and if the maps $\varphi_i$ are isomorphisms for $i\geq m$ for some integer $m$.}
\end{Definition}

Let $\mathcal{M}=\{M_i, \varphi_i\}_{i\geq 0}$ be a persistence module.
We assign to $\mathcal{M}$ a graded module over the graded ring $\mathbb{F}[x]$ as follows.
Let $$\alpha(\mathcal{M})=\bigoplus_{i=0}^{\infty}M_i$$
be the direct sum of the structures on the individual components. Place a graded $\mathbb{F}[x]$-module structure on $\alpha(\mathcal{M})$ with $x$ acting as a shift map. More precisely,
$$x\cdot (m_0, m_1,\cdots)=(0, \varphi_0(m_0), \varphi_1(m_1),\cdots).$$
It is known that the assignment $\alpha$ defines an equivalence of categories between the category of persistence modules of finite type over $\mathbb{F}$ and the category of finitely generated nonnegatively graded $\mathbb{F}[x]$-modules.

The graded ring $\mathbb{F}[x]$ is a PID and its only possible graded ideals are of the form $(x^n)=x^n\cdot\mathbb{F}[x], n\geq 0.$ Then the classification of $\mathbb{F}[x]$-modules follows from the structure theorem for PID's. This implies the following theorem.

\begin{Theorem}\label{structure}
Let $\mathcal{G}=\{G_i, f_i\}$ be a persistence graph of finite type. Then for
the persistence module in (\ref{persistence module}),
\begin{equation}\label{iso}
\bigoplus_{i=0}\mathbb{F}\aut(G_i)\cong \left(\bigoplus_{i=0}\Sigma^{t_i}\mathbb{F}[x]\right)\oplus\left(\bigoplus_{j=0}\Sigma^{r_j}\left(\mathbb{F}[x]/(x^{s_j})\right)\right),
\end{equation}
where $\Sigma^d$ denotes an $d$-shift upward in grading.
\end{Theorem}
This classification theorem has a natural interpretation. The free portion of (\ref{iso}) are in bijective correspondence with those automorphism group generators which come into existence at parameter $t_i$ and which are still alive for all future parameter values.
The torsion elements correspond to those automorphism group generators that appear at parameter $r_j$ and vanish at parameter $r_j+s_j$.

Before proceeding any further, we parametrize the isomorphism classes of $\mathbb{F}[x]$-modules by suitable intervals.
\begin{Definition}
{\rm A \emph{$\mathcal{P}$-interval} is an ordered pair $(i, j)$ with $0\leq i<j,$ where $i, j\in\mathbb{Z}\cup\{+\infty\}.$}
\end{Definition}
We associate a graded $\mathbb{F}[x]$-module to a set  $\mathcal{S}$ of $\mathcal{P}$-intervals via a correspondence $Q$ given in the following way. Let
\begin{align*}
&Q(i, j)=\Sigma^{i}\left(\mathbb{F}[x]/(x^{j-i})\right),\\
&Q(i, +\infty)=\Sigma^{i}\mathbb{F}[x].
\end{align*}
For a set of $\mathcal{P}$-intervals $\mathcal{S}=\{(i_1, j_1),\cdots, (i_m, j_m)\}$, define $$Q(\mathcal{S})=\bigoplus^m_{t=0}Q(i_t, j_t).$$
\begin{Definition}
{\rm A finite set of $\mathcal{P}$-intervals is called a \emph{barcode}.}
\end{Definition}

One can easily see that the correspondence $Q$ defines a bijection between the finitely generated graded $\mathbb{F}[x]$-modules and the barcodes. With this correspondence,
Theorem \ref{structure} yields the fundamental characterization of barcodes.
\begin{Theorem}
The order of  $\aut(G_i)^{i, i+p}$ is equal to the number of intervals in the barcode of $\bigoplus_{i=0}\mathbb{F}\aut(G_i)$ spanning the interval $[i, i+p]$. In particular, the order of $\aut(G_i)$ is equal to the number of intervals containing $i$.
\end{Theorem}

Barcodes are a very intuitive way of representing the evolution of the automorphism groups.
Using barcodes, one can provide a visual description of the evolution of the automorphism groups in a filtration of a finite graph, as illustrated in the following example.

\begin{Example}\label{example:graph_filtration}
{\rm Let $G=C_4$, the cycle of length $4$. Choose a labelling of the $4$ vertices, as shown in Figure \ref{graph1}. Consider a filtration of $C_4$ given by
\[
G_0 \xhookrightarrow{f_0} G_1 \xhookrightarrow{f_1} G_2 \xhookrightarrow{f_2} G_{\geq 3}=G.
\]

where
\begin{align*}
&G_0=\{1, 2, 3, 4\},\\
&G_1=\{1, 2, 3, 4,  \{2, 3\}\},\\
&G_2=\{1, 2, 3, 4, \{1, 4\}, \{2, 3\}\},
\end{align*}
and
$f_i: G_i\rightarrow G_{i+1}$ is the natural inclusion map.
 Then we obtain  a persistence graph $\mathcal{G}=\{G_i, f_i\}$.
\begin{figure}[htbp]
\centering
  \includegraphics[width=0.8\textwidth]{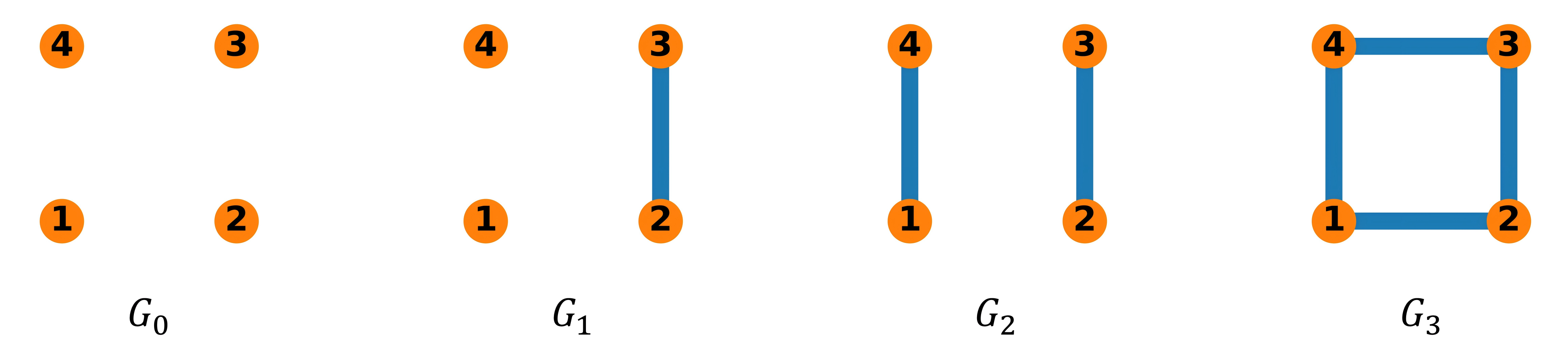}
\caption{Illustration of the filtration of graphs in Example \ref{example:graph_filtration}. }
	\label{graph1}
\end{figure}
It is known that the automorphism group of $C_4$ is the dihedral group $D_8$ of order $8$.
Let $r$ be the rotation clockwise through $\frac{\pi}{2}$ radian and
let $s$ be the reflection about the line of symmetry through vertex $1$ and vertex $3$.
Then $$D_8=\{1, r, r^2, r^3, s, sr, sr^2, sr^3\}.$$
A straightforward calculation yields
\begin{align*}
&\mathbb{F}{\rm Aut}(G_0)^{0, 0}=\mathbb{F}S_4,\\
&\mathbb{F}{\rm Aut}(G_0)^{0, 1}=\mathbb{F}\{1,  (14), (23), sr=(14)(23)\},\\
&\mathbb{F}{\rm Aut}(G_0)^{0, 2}=\mathbb{F}\{1,  (14), (23), sr=(14)(23) \},\\
&\mathbb{F}{\rm Aut}(G_0)^{0, 3}=\mathbb{F}\{1  \},\\
&\mathbb{F}{\rm Aut}(G_1)^{1, 1}=\mathbb{F}\{1,  (14), (23), sr=(14)(23)\},\\
&\mathbb{F}{\rm Aut}(G_1)^{1, 2}=\mathbb{F}\{1,  (14), (23), sr=(14)(23)\},\\
&\mathbb{F}{\rm Aut}(G_1)^{1, 3}=\mathbb{F}\{1 \},\\
&\mathbb{F}{\rm Aut}(G_2)^{2, 2}=\mathbb{F}\{1,  (14), (23), sr=(14)(23), r^2=(13)(24), sr^3=(12)(34)\},\\
&\mathbb{F}{\rm Aut}(G_2)^{2, 3}=\mathbb{F}\{1,  sr^3=(12)(34)\},\\
&\mathbb{F}{\rm Aut}(G_3)^{3, 3}=\mathbb{F}D_8,\\
\end{align*}

Thus the corresponding barcode is $\mathcal{S}=$
\begin{align*}
&\{\underbrace{[0, 1),\cdots, [0, 1)}_{20 \text{ terms}},~[0, +\infty),~[0, 3),~ [0, 3),~[0, 3)\}\\
\cup&\{[1, +\infty), ~ [1, 3), ~ [1, 3), ~ [1, 3)\}\\
\cup&\{[2, +\infty),~[2, +\infty), ~ [2, 3), ~ [2, 3), ~ [2, 3)~ [2, 3)\}\\
\cup&\{\underbrace{[3, +\infty),\cdots, [3, +\infty)}_{8 \text{ terms}}\}.
\end{align*}

Figure \ref{graph2} gives a graphical representation of the barcode $\mathcal{S}$ as a collection of horizontal line segments in a plane whose horizontal axis corresponds to the parameter and whose vertical axis represents an ordering of the generators of automorphism groups.}

\begin{figure}[htbp]
	\centering
	\includegraphics[scale=0.75]{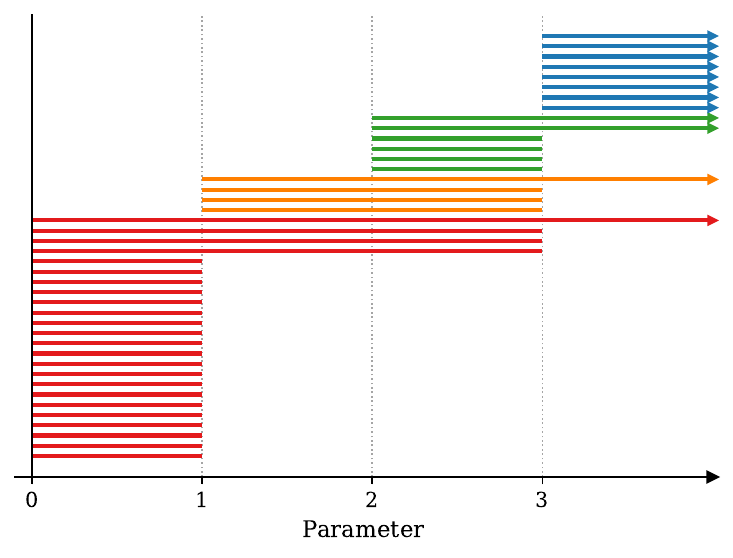}
	\caption{The barcode for persistence module $\{\mathbb{F}{\rm Aut}(G_i), \Phi(f_i)\}_{i\geq 0}$. }
	\label{graph2}
\end{figure}
\end{Example}

\subsection{Two criteria}

Recall that in the definition of persistent automorphism module, the key ingredient is that when $G^{\prime}=(V(G^{\prime}), E(G^{\prime}))$ is a subgraph of $G=(V(G), E(G))$, we need to consider
the homomorphism $$\mathbb{F}{\rm Aut}(G) \rightarrow \mathbb{F}{\rm Aut}(G^{\prime})$$ induced by the restriction.
In other words, we are interested in the automorphisms of $G^{\prime}$  which are the restrictions of automorphisms of $G$.

In what follows, we shall restrict our attention to the case where
the subgraph $G^{\prime}$ is required to have the property that
\begin{equation}\label{subgraph}
V(G^{\prime})=V(G),~ E(G^{\prime})\subseteq E(G).
\end{equation}
This assumption imposed on the subgraph $G^{\prime}$ is not just an idle technicality. In fact,
this case arises naturally in practice, which, as we shall explain, validates the assumption we make.

Given a collection of points $X=\{x_{\alpha}\}$ in a metric space, the obvious way to convert $\{x_{\alpha}\}$ into a global object is to use the point cloud as the vertices of a graph whose edges are determined by proximity. One of the most natural methods for doing so is to apply  \emph{Vietoris-Rips} complex $\mathcal{R}_\varepsilon(X)$.  Proposition \ref{proposition:equation} says $$\mathcal{R}_\varepsilon(X) = F(G_\varepsilon(X)).$$
When $0<\varepsilon_1<\varepsilon_2$,
$G_{\varepsilon_1}(X)$ is a subgraph of $G_{\varepsilon_2}(X)$ satisfying (\ref{subgraph}).
This justifies the condition (\ref{subgraph}).

The main goal of this subsection  is to present two necessary conditions concerning whether the automorphisms of the subgraph $G^{\prime}$ are the restrictions of automorphisms of $G$.

Let $G=(V(G), E(G))$ be a finite simple graph. In the remainder of this subsection, the symbols $H_p$ and $H^p$ will denote simplicial homology and simplicial cohomology, respectively. For a deeper exposition of the related algebraic topology, the reader is referred to {\rm{\cite{AH2002,JRM2018,CA2013}}}.

A \emph{cycle} of a graph is a path that starts and ends at the same vertex, with no other repeated vertices, and all edges in the path are distinct.

\begin{Definition}
	{\rm A cycle $P$ of a graph $G$ is called \emph{automorphism-invariant} in $G$ if every automorphism of $G$ maps $P$ to itself, i.e., $\eta(P)=P$ for all $\eta\in{\rm \aut}(G)$.}
\end{Definition}

\begin{Definition}
	{\rm A graph $G$ is called \emph{cycle-stable} if all its cycles are automorphism-invariant.}
\end{Definition}

\begin{Example}{\rm Figure (\ref{fig:exmaple2}) depicts a cycle-stable graph $G$. It is clear that the cycles $(v_2,v_3,v_4)$, $(v_0,v_1,v_2,v_3)$ and $(v_1,v_2,v_5,v_6,v_7)$ are
cycle-stable.
	\begin{figure}[htbp]
		\centering
         \includegraphics[width=0.6\textwidth]{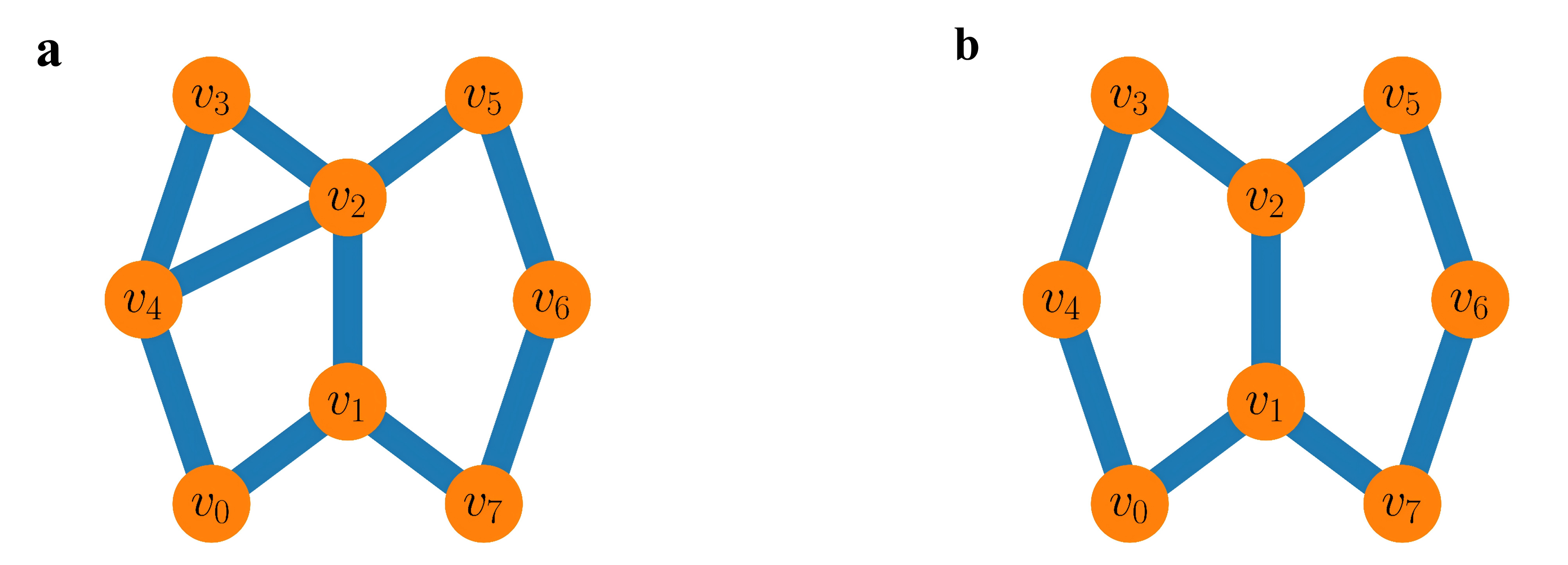}\\
		\caption{\textbf{a} Graph $G$ with three automorphism-invariant cycles; \textbf{b} Graph $G^{\prime}$ no automorphism-invariant cycle.}\label{fig:exmaple2}
	\end{figure}
However, the subgraph $G^{\prime}$ with $$E(G^{\prime})=E(G)-\{v_2v_4\}$$ fails to satisfy the automorphism-invariant property and is not cycle-stable. In fact, assigning to $v_3, v_4, v_0$ to $v_5, v_6, v_7$ respectively and keeping $v_1, v_2$ fixed gives rise to an automorphism of $G^{\prime}$ that sends cycle $(v_0, v_4, v_3,v_2,v_1)$ to cycle $(v_7,v_6,v_5,v_2,v_1)$.

}	

\end{Example}

Let $\omega$ denote the generator of the free cyclic group $H_1(S^1)$ determined by the counter-clockwise orientation of $S^1$. By means of simplicial homology,We establish the first criterion for when automorphisms of $G$ restrict to automorphisms of $G^{\prime}$.

\begin{Theorem}\label{criterion1}
Let $G=(V(G), E(G))$ be a finite simple and cycle-stable graph. Let $G^{\prime}=(V(G^{\prime}), E(G^{\prime}))$ be a subgraph of $G$ such that $$V(G^{\prime})=V(G),~ E(G^{\prime})\subseteq E(G).$$
Let $\eta\in {\rm Aut}(G)$.  If
$\eta$ restricts to an automorphism ${\eta|_{G^{\prime}}}: G^{\prime}\rightarrow G^{\prime}$, then for any given map $f: S^1\rightarrow |G^{\prime}|$,
$$(|\eta|_{G^{\prime}}|\circ f)_*(\omega)=\pm(f_*\omega)$$ in $H_1(|G^{\prime}|)$, where $G^{\prime}$ is viewed as a $1$-dimensional simplicial complex.
\end{Theorem}

\begin{proof}
A key observation is that  $H_1(|G^{\prime}|)$ is a free abelian group generated by all cycles of $G^{\prime}$, since the complex $G^{\prime}$ is $1$-dimensional.
According to the hypothesis that $G$ is cycle-stable, we obtain
$$\eta|_{G^{\prime}}(P)=P$$
 for every cycle $P$ of $G^{\prime}$. The proof of this theorem is finished.
\end{proof}

As we shall see in Theorem \ref{criterion}, our second criterion has a description in terms of simplicial cohomology that is very similar to, and in a certain sense dual to, the
first one. However, the proof of this criterion is much more complicated.
Before proceeding further we need to verify a technical property.
\begin{Lemma}\label{maps to 1}
Any map $f: |F(G)|\rightarrow S^1$ is homotopic to a map which sends $V(G)$ to the point $s_0$ of $S^1$.
\end{Lemma}
\begin{proof}
 The restriction $$f|_{V(G)}: V(G)\rightarrow S^1$$ being not surjective implies there exists a homotopy  $$h_t: V(G)\rightarrow S^1, ~0\leq t\leq 1,$$ such that $h_0=f|_{V(G)},$ and $h_1(x)=s_0$ for every $x\in V(G)$.
Consider the product space $M=|F(G)|\times I$ and its closed subspace $L=(|F(G)|\times 0)\cup (V(G)\times I)$. Define a map $H: L\rightarrow S^1$ by setting
\begin{equation*}
H(x, t)=
\begin{cases}
f(x),~&\text{if}~x\in |F(G)|,~ t=0,\\
h_t(x),~&\text{if}~x\in A,~ t\in I.
\end{cases}
\end{equation*}
Since there is a retract $r: M\rightarrow L$, we obtain a homotopy $g_t,~ 0\leq t\leq 1$ by taking $$g_t(x)=H\circ r(x, t)$$ for every $(x, t)\in M$.  $g_t$ is obviously an extension of $h_t$ such that $g_0=f$. Then the map $g_1$ is the desired one  mapping $V(G)$ into $s_0$.
\end{proof}

We now begin the discussion of the second criterion.
Let $\theta$ denote the generator of the free cyclic group $H^1(S^1)$ determined by the counter-clockwise orientation of $S^1$.

\begin{Theorem}\label{criterion}
Let $G=(V(G), E(G))$ be a finite simple graph with all cycles automorphism-invariant
and  $G^{\prime}=(V(G^{\prime}), E(G^{\prime}))$ a subgraph of $G$ such that $$V(G^{\prime})=V(G),~ E(G^{\prime})\subseteq E(G).$$
Let $\eta\in {\rm Aut}(F(G))$.  If
$\eta$ restricts to an automorphism $\eta|_{G^{\prime}}: F(G^{\prime})\rightarrow F(G^{\prime})$, then when given any  map $f: |F(G)|\rightarrow S^1$,
 $$(f\circ|\eta|_{G^{\prime}}|)^*(\theta)(P)=\pm(f|_{|F(G^{\prime})|})^*(\theta)(P)$$ for any cycle $P$ in $H_1(|F(G^{\prime})|)$.
\end{Theorem}
\begin{proof}
By Lemma \ref{maps to 1}, we may assume that $f$ takes all vertices of $F(G)$ to $s_0$.
For each $f\circ|\xi|$, $\xi\in {\rm Aut}(F(G^{\prime}))$, we shall  construct a cochain $$c^1(\xi)\in C^1(|F(G^{\prime})|)$$ such that $c^1(\xi)$ represents $(f\circ|\xi|)^*(\theta)$ in $H^1(|F(G^{\prime})|).$

Let $\sigma=v_0v_1$ be an arbitrary $1$-simplex in $F(G^{\prime})$. Denote by $\phi_{\sigma}: \Delta_1\rightarrow\sigma$ the linear map sending $0$ to $v_0$ and $1$ to $v_1$. Note that the composition
$$f\circ|\xi|\circ\phi_{\sigma}$$ is a loop in $S^1$. Thus the degree of $f\circ|\xi|\circ\phi_{\sigma}$ is defined. Suppose that $\gamma_1$ is a generator of  $H_1(\Delta_1, \{0, 1\})$ satisfying that $$\sigma=(\phi_{\sigma})_*(\gamma_1).$$ Then $\{\sigma\}$ forms a basis for $H_1(|F(G^{\prime})|^1, |F(G^{\prime})|^0)$ as $\sigma$ ranges over the $1$-simplices of $F(G^{\prime})$, where $\phi_{\sigma}$ is the characteristic map for $\sigma$, $F(G^{\prime})$ being viewed as a CW complex.
Define the cochain $c^1(\xi)\in C^1(|F(G^{\prime})|)$ by
\begin{equation}\label{representative}
c^1(\xi)(\sigma)={\rm deg}(f\circ|\xi|\circ\phi_{\sigma}).
\end{equation}

$c^1(\xi)$ is indeed a cocycle. To prove this assertion, let $\tau=v_0v_1v_2$ be an arbitrary $2$-simplex in $F(G^{\prime})$. Then
$$\partial\tau=\sigma_0-\sigma_1+\sigma_2,$$
where $\sigma_0=v_1v_2,~ \sigma_1=v_0v_2,~ \sigma_2=v_0v_1$.
Let $\phi_{\tau}: \Delta_2\rightarrow\tau$ be the linear homeomorphism which preserves the order of vertices. Then
$$\tau=(\phi_{\tau})_*(\gamma_2),$$
for some generator $\gamma_2$ of $H_2(\Delta_2, \partial\Delta_2)$. It follows that
\begin{align}\label{cocycle}
(\delta c^1(\xi))(\tau)&=c^1(\xi)(\partial\tau)=c^1(\xi)(\sigma_0)-c^1(\xi)(\sigma_1)+c^1(\xi)(\sigma_2)\\
&={\rm deg}(f\circ|\xi|\circ\phi_{\sigma_0})-{\rm deg}(f\circ|\xi|\circ\phi_{\sigma_1})+{\rm deg}(f\circ|\xi|\circ\phi_{\sigma_2}).\nonumber
\end{align}
It remains to show that $(\ref{cocycle})$ is equal to $0$. To see this, consider the restriction $$g=f\circ|\xi|\circ\phi_{\tau}|_{\partial\Delta_2}: \partial\Delta_2\rightarrow S^1.$$ Denote $\partial\Delta_2=v_0^{\prime}v_1^{\prime}v_2^{\prime}$.
Then the generator for $H_1(\partial\Delta_2)$ is $\sigma_0^{\prime}-\sigma^{\prime}_1+\sigma^{\prime}_2$, where $\sigma^{\prime}_0=v^{\prime}_1v^{\prime}_2,~ \sigma^{\prime}_1=v^{\prime}_0v^{\prime}_2,~ \sigma^{\prime}_2=v^{\prime}_0v^{\prime}_1$.
Observe that
$$\sigma^{\prime}_j=(\phi_{\sigma^{\prime}_j})_*(\gamma_1)~\text{and}~ (\phi_{\tau}\circ\phi_{\sigma^{\prime}_j})_*(\gamma_1)=(\phi_{\sigma_j})_*(\gamma_1).$$
Hence \begin{align*}
{\deg}g&=g^*(\theta)(\sigma_0^{\prime}-\sigma^{\prime}_1+\sigma^{\prime}_2)\\
&=\theta(g_*(\sigma_0^{\prime}-\sigma^{\prime}_1+\sigma^{\prime}_2))\\
&={\rm deg}(f\circ|\xi|\circ\phi_{\sigma_0})-{\rm deg}(f\circ|\xi|\circ\phi_{\sigma_1})+{\rm deg}(f\circ|\xi|\circ\phi_{\sigma_2})\\
&=(\ref{cocycle}).
\end{align*}
The fact $g$ is the restriction of $f\circ|\xi|\circ\phi_{\tau}$ which is defined on $\Delta_2$ means that $g$ is homotopic to a constant map. It follows that ${\deg}g=0$, finishing the proof that $c^1(\xi)$ is a cocycle.

We claim that $c^1(\xi)$ is a representative of $(f\circ|\xi|)^*(\theta)$ in $H^1(|F(G^{\prime})|).$ To see this, by direct computation, we have for an arbitrary $1$-simplex $\sigma$ in $F(G^{\prime})$,
\begin{align*}
(f\circ|\xi|)^*(\theta)(\sigma)&=\theta((f\circ|\xi|)_*(\sigma))\\
&=\theta((f\circ|\xi|)_*((\phi_{\sigma})_*(\gamma_1))\\
&={\rm deg}(f\circ|\xi|\circ\phi_{\sigma})\\
&=c^1(\xi)(\sigma).~(\text{by (\ref{representative})})
\end{align*}

 Suppose that
\begin{align*}
P=(v_0, v_1,\cdots, v_m)
\end{align*}
is a cycle in $G^{\prime}$.

A key observation is that  $$\eta|_{G^{\prime}}(P)=P,$$   the cycle $P$ being automorphism-invariant in $G$.
Note that from $(\ref{representative})$, we know that
\begin{align*}\label{well-defined}
&c^1({\rm id}_{G^{\prime}})(v_{0}v_{1})+\cdots+c^1({\rm id}_{G^{\prime}})(v_{{m-1}}v_{{m}})+c^1({\rm id}_{G^{\prime}})(v_{{m}}v_{0})\nonumber\\
=&\pm (c^1(\eta|_{G^{\prime}})(v_{0}v_{1})+\cdots+c^1(\eta|_{G^{\prime}})(v_{{m-1}}v_{{m}})+c^1(\eta|_{G^{\prime}})(v_{{m}}v_{0})),
\end{align*}
since $\eta|_{G^{\prime}}$ and ${\rm id}_{G^{\prime}}$ are the automorphisms of $P$, which induce  permutations of edges of $P$. It follows that
\begin{align*}
c^1({\rm id}_{G^{\prime}})(P)=\pm c^1(\eta|_{G^{\prime}})(P).
\end{align*}
Note that $c^1(\eta|_{G^{\prime}})$ and $c^1({\rm id}_{G^{\prime}})$ represent $(f\circ|\eta|_{G^{\prime}}|)^*(\theta)$ and $(f|_{|F(G^{\prime})|})^*(\theta)$ respectively in $H^1(|F(G^{\prime})|).$
Therefore, we conclude that $(f\circ|\eta|_{G^{\prime}}|)^*(\theta)(P)=\pm(f|_{|F(G^{\prime})|})^*(\theta)(P)$.

\end{proof}

\begin{Example}
{\rm Consider the graph $G$  described in Figure \ref{figure:example2}. Let $G^{\prime}$ be the subgraph given by $$E(G^{\prime})=E(G)-\{v_1v_2, v_1v_3\}.$$
\begin{figure}[htbp]
	\centering
         \includegraphics[width=0.6\textwidth]{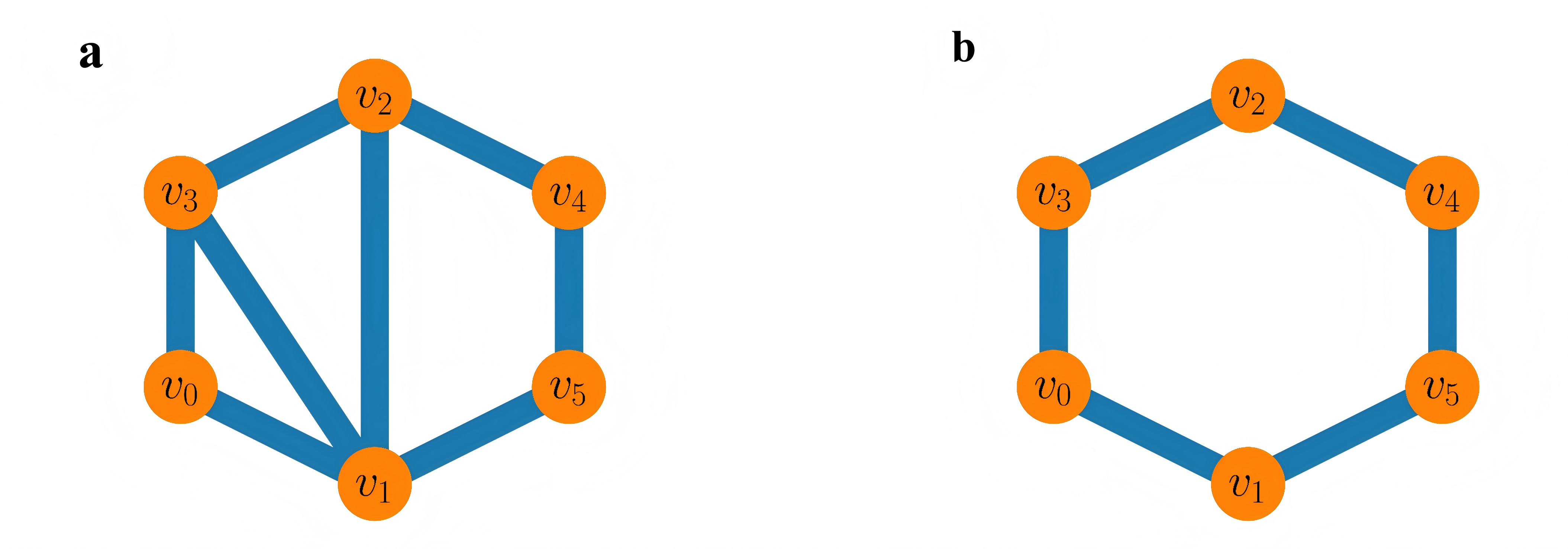}\\
   \caption{\textbf{a} Graph $G$ which is cycle-stable;  \textbf{b} Graph $G'$ with only one cycle.}\label{figure:example2}
\end{figure}
Identify $S^1$ with the cycle $(v_0, v_5, v_4, v_3)$. Let $\eta\in {\rm Aut}(G)$ be the permutation defined by $(v_0v_5)(v_3v_4)$. Then clearly $\eta|_{G^{\prime}}$ is an automorphism of $G^{\prime}$.
Consider the map
$f: S^1\rightarrow |G^{\prime}|$ given by the inclusion. Then a direct computation gives
$$(|\eta|_{G^{\prime}}|\circ f)_*(\omega)=-(f)_*(\omega)$$ in $H_1(|G^{\prime}|)$.

The same conclusion can also be derived from Theorem \ref{criterion}. To see this, let us consider the map $g: |F(G)|\rightarrow S^1$ defined by the linear map determined by mapping
$v_i$ to $v_i$ for $i=0,3,4,5$, $v_1$ to $v_0$, and $v_2$ to $v_3$. Let $P$ be the unique  cycle $(v_0,v_1,v_5,v_4,v_2,v_3)$.
An easy computation shows that
$$(g\circ|\eta|_{G^{\prime}}|)^*(\theta)(P)=-(g|_{|F(G^{\prime})|})^*(\theta)(P).$$
}
\end{Example}

\section{Structure and stability analysis of fullerene molecules}\label{section:application}

As a key technique for quantifying multiscale topological features such as connectivity and voids, persistent homology has found extensive use in studying the architecture of biological molecules, see references {\rm{\cite{GC2009,GC2005,HE2010,HE2002,RG2008,AZ2005}}}.
 In this subsection, the theory and algorithms of persistent automorphism modules are employed to study the structure  of fullerene $\mathrm{C}_{60}$.

\subsection{Structure analysis of  fullerene $\mathrm{C}_{60}$}

Recall that the order of the automorphism group $\aut(G)$ of $G$ serves as a quantitative indicator of the symmetry of a graph $G$, a large order implying $G$ is highly symmetric.
Given the ground-state structural data $X$ of a fullerene molecule, which contains coordinates of fullerene carbon atoms, the collection of atom center locations forms a point cloud in  Euclidean space $\mathbb{R}^3$. For each $\varepsilon>0$, we have the corresponding proximity graph $G_\varepsilon(X)$. When the distance threshold $\varepsilon$ is very small, smaller than the shortest $\mathrm{C}$-$\mathrm{C}$ bond length, the resulting graph $G_\varepsilon(X)$ is  disconnected, and its automorphism group order  reflects trivial vertex permutations.

As $\varepsilon$ increases to  certain critical bond lengths, many edges are formed within a short interval. This causes abrupt changes in local connectivity, breaking previously symmetric neighborhoods of some carbon atoms. Consequently, the order of automorphism group  drops sharply. When $\varepsilon$ increases further, more edges appear and the overall connectivity becomes more regular, restoring or even enhancing the global symmetry of this molecule. This leads to a rapid increase in the order of $\aut(G_\varepsilon(X))$.

Before proceeding to a further discussion, we first give some notions that will be used in the sequel.
\begin{Definition}
	{\rm Let $G$ be a simple and connected graph. The \emph{symmetry degree} of $G$ is defined to be
		\begin{equation*}
		\Gamma(G) = \sum_{\sigma\in\aut(G)}{\rm ord}(\sigma).
		\end{equation*}
		If $G$ is not connected, then the symmetry degree of $G$ is defined by
		\begin{equation*}
		\Gamma(G)=\sum_{i=1}^n\Gamma(G_i),
		\end{equation*}
		where $G_1,\cdots, G_n$ are path-components of $G$.}
\end{Definition}

\begin{Definition}
	{\rm Let $X$ be a set of points in Euclidean space $\mathbb{R}^3$. The \emph{symmetry order curve} of $X$ is the function
		$$\delta(\varepsilon)=\log_2(|\aut(G_\varepsilon(X))|),$$
		where $|{\aut}(G_\varepsilon(X))|$ denote the order of the automorphism group of $G_\varepsilon(X)$ at the distance $\varepsilon$.
	}
\end{Definition}
In a similar manner, we introduce the notion of the symmetry degree curve.
\begin{Definition}
	{\rm Let $X$ be a set of points in Euclidean space $\mathbb{R}^3$. The \emph{symmetry degree curve} of $X$ is the function
		$$\gamma(\varepsilon)=\log_2(\Gamma(G_\varepsilon(X))),$$
		where $\Gamma(G_\varepsilon(X))$ is the symmetry degree of $G_\varepsilon(X)$ at the distance $\varepsilon$.
	}
\end{Definition}

We do not consider the symmetry degree of isolated points in the calculations, as studying the symmetry of isolated points is of limited significance. This simplification does not affect the characterization of symmetry.

We will use fullerene $\mathrm{C}_{60}$ as an example to provide more explicit demonstration of the previous discussion. For a given sampling points of fullerene $\mathrm{C}_{60}$, Figure \ref{graph3} shows symmetry order curve and symmetry degree curve at distance parameter $0<\varepsilon\leq 3$, where from top to bottom, the behaviors of $\delta(\varepsilon)$ and  $\gamma(\varepsilon)$ are depicted. In this figure, we find that the behavior of $\delta(\varepsilon)$ exhibits significant variations approximately at the points 1.4, 2.2, 2.4 and 2.7.

\begin{figure}[htbp]
	\centering
	\includegraphics[width=0.9\textwidth]{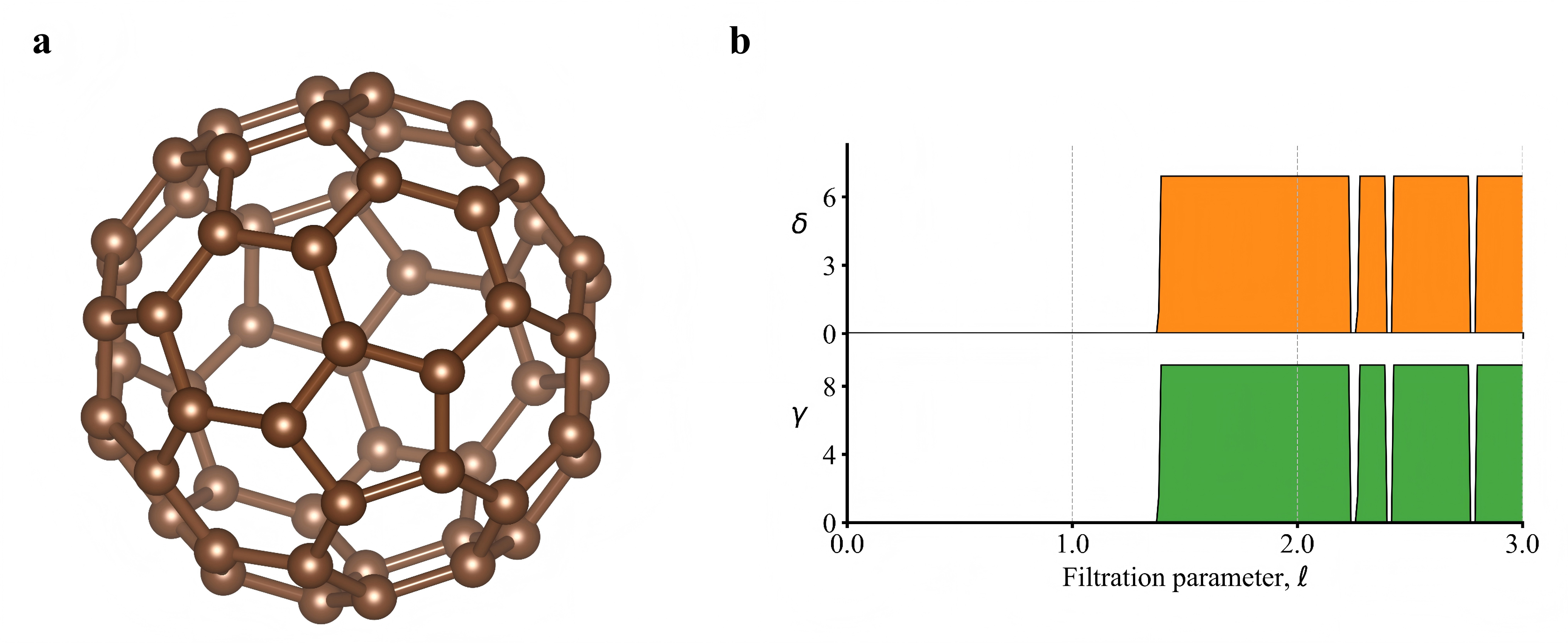}
	\caption{\textbf{a} Illustration of the structure of fullerene $\mathrm{C}_{60}$; \textbf{b} The symmetry order curve and the symmetry degree curve for fullerene $\mathrm{C}_{60}$.}
	\label{graph3}
\end{figure}

 To elucidate the underlying mechanism, let us analyze the carbon atom connectivity in fullerene $\mathrm{C}_{60}$ at different length scales as follows.
 \begin{itemize}
 	\item Nearest neighbors (chemical bonds): $1.40 \mathrm{\AA}-1.46 \mathrm{\AA}$.
 	\item Second nearest neighbors: $2.2 \mathrm{\AA}-2.5 \mathrm{\AA}$.
 	\item Third nearest neighbors: $2.7 \mathrm{\AA}-2.9 \mathrm{\AA}$.
 \end{itemize}
With the threshold $\varepsilon$ increasing, the order of automorphism group undergoes multiple drop-rise intervals, which are
$$(1.4 \mathrm{\AA}, 1.45 \mathrm{\AA}),\quad (2.2 \mathrm{\AA}, 2.45 \mathrm{\AA}),\quad (2.7 \mathrm{\AA}, 2.75 \mathrm{\AA}).$$
For the first drop-rise interval $(1.4 \mathrm{\AA}, 1.45 \mathrm{\AA})$,  all nearest-neighbor bonds are established at this stage, but second nearest neighbors remain largely unconnected. This uneven local connectivity temporarily disrupts local symmetry, resulting in a sharp decrease in the automorphism group order. As some second nearest neighbors are progressively incorporated, local symmetry is restored, causing the order to rise. For the second drop-rise interval $(2.2 \mathrm{\AA}, 2.5 \mathrm{\AA})$, the inclusion of additional second nearest neighbors occurs while some third nearest neighbors are still absent. The imbalance in local connectivity temporarily breaks symmetry, leading to another decline in group order. Once the threshold reaches about $2.5 \mathrm{\AA}$, adjacency among all second nearest neighbors is complete, and the automorphism group order increases again. For the third drop-rise interval $(2.7 \mathrm{\AA}, 2.9 \mathrm{\AA})$, partial connections with third nearest neighbors create nonuniform local neighborhoods, reducing symmetry and lowering the group order. As the threshold further increases, the graph achieves full connectivity, restoring the overall symmetry and raising the order once more.

In summary, each drop-rise interval corresponds to the stepwise incorporation of neighbors at a specific distance scale, temporarily disrupting local symmetry. The multiple intervals appearing in fullerene $\mathrm{C}_{60}$ reflect its large size and complex geometric structure, where local and global symmetries emerge at different $\varepsilon$.

As we have seen, the interpretation of this phenomenon can be provided through the lens of persistent automorphism module, as it captures the multiscale evolution of topological features and reveals how local symmetries are gradually disrupted and subsequently restored with the threshold increasing.

Since the hexagon is a fundamental building block of fullerene $\mathrm{C}_{60}$, we consider a model in which the hexagon is represented as a graph and subjected to a scale-dependent filtration, where edges are included progressively according to an increasing adjacency threshold. This allows us to more directly and intuitively reveal the underlying structural mechanisms.
\begin{figure}[htbp]
	\centering
	\includegraphics[width=0.8\textwidth]{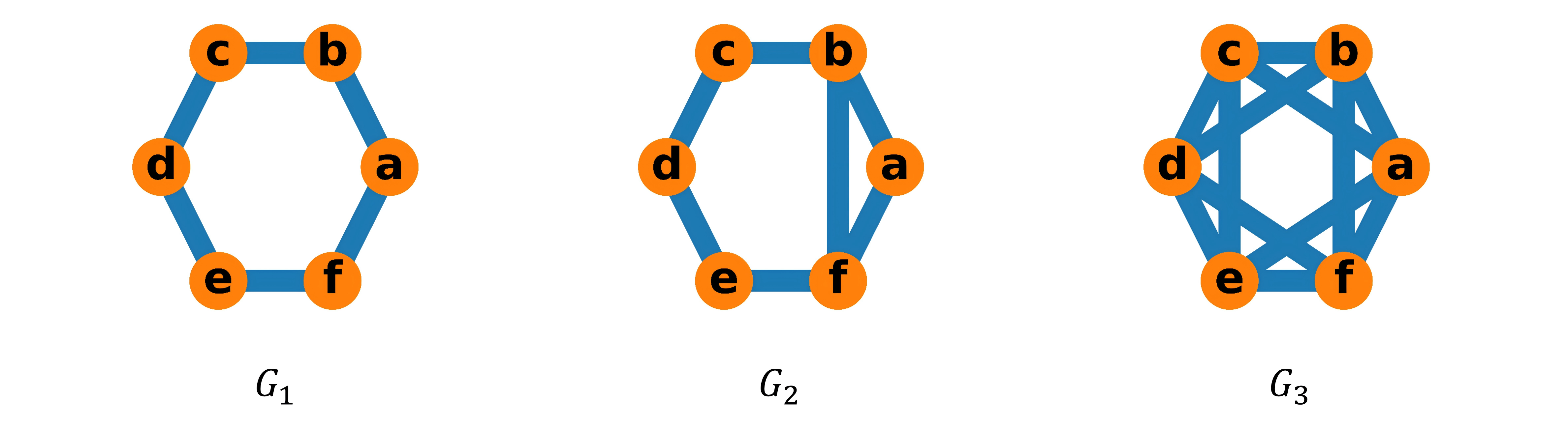}
	\caption{Illustration of a persistence graph of the hexagon.}
	\label{graph4}
\end{figure}

Figure \ref{graph4} illustrates a persistence graph $$\mathcal{G}=\{G_i, f_i\}_{i=1}^3,$$
where $f_i$ is the constant map for each $i$. Consider the map $$\aut(G_2) \rightarrow \aut(G_1)$$ induced by $G_1 \to G_2$. Note that $G_1$ has enough automorphisms, and all the automorphisms of $G_2$ can persist from $\aut(G_2)$ to $\aut(G_1)$. In comparison to $G_3$, the graph $G_2$ is relatively asymmetric. Consequently, many automorphisms of $G_3$, such as the nontrivial rotations, cannot persist from $\aut(G_3)$ to $\aut(G_2)$. This leads to a decrease in symmetry order or symmetry degree.

\subsection{Stability prediction of small fullerene molecules}
As shown by the preceding analysis of the structure of $\mathrm{C}_{60}$, it can be seen that detailed symmetry structural information of molecules has been encoded by the corresponding persistent automorphism modules. On the other hand, molecular structure plays a decisive role in determining molecular function. This provides the possibility that persistent automorphism modules of graphs can be used to predict the stability of small fullerene molecules. In this section, we consider the information of barcode lengths in the  persistence diagrams.

It has been observed that the ground-state heat of formation of fullerene molecules tends to decrease with an increasing number of atoms. Our study relies on the data sets reported in {\rm{\cite{BL1992}}}. Previous work {\rm{\cite{BLZ1992}}} suggested that the stability of fullerenes is influenced by the ratio between the number of pentagons and the total number of atoms in the molecule. Subsequently, K. Xia and colleagues {\rm{\cite{KX2015}}} applied persistent homology using Vietoris-Rips complexes as a tool to predict fullerene stability. More recently, the study in {\rm{\cite{JL2022}}} employed neighborhood hypergraph to estimate fullerene stability. All of these approaches emphasize the importance of the molecular structure in determining stability.

In the study of fullerene and related molecular graphs, the order of the automorphism group provides a rigorous quantitative measure of structural symmetry. A large automorphism group order corresponds to a high degree of symmetry in the carbon cage, such as the well known icosahedral configuration of $\mathrm{C}_{60}$, whereas smaller orders indicate lower symmetry and more irregular structures. This distinction is not merely combinatorial since it has direct implications for molecular behavior. High symmetry fullerenes tend to exhibit greater structural stability, reduced strain energy, and enhanced electronic degeneracy, while less symmetric structures often correspond to higher reactivity and less favorable energetic profiles.

Therefore, the order of the automorphism group functions as a bridge between geometry and molecular chemistry. It encodes, in a single algebraic invariant, the extent to which the connectivity pattern of carbon atoms admits nontrivial permutations. By comparing the group orders across different fullerene graphs, one can classify molecular isomers, assess their relative symmetries, and predict qualitative trends in stability and physical properties. In this sense, the order of the automorphism group is not only a mathematical invariant but also a structural descriptor of profound significance in the analysis and understanding of fullerene molecules.

As we shall see, persistent automorphism module serves as an alternative description of fullerene structures and are applied in the analysis of their predicted stability.

In this work, we analyze the structure and stability of fullerene molecules using the symmetry order curve and symmetry degree curve based on their atomic coordinates. We hypothesize that the stability of fullerene molecules is not only related to the number of atoms but also strongly correlated with their symmetry. The symmetry order curve and symmetry degree curve represent the quantitative changes in symmetry. By observing Figure \ref{best}, we note that each fullerene molecule exhibits a relatively long and stable segment in both its symmetry order curve and symmetry degree curve. This stable phase indicates that the fullerene maintains a relatively stable symmetric structure. Our preliminary approach is to use the duration of this stable phase to analyze the fullerene structure.

For enhanced computational stability and analytical effectiveness, we approximate the length of this stable segment by considering the non-zero parts before and after the stable phase and calculating an average length. Specifically, we compute the integral area of this region and divide it by the vertical value of the stable part, obtaining the corresponding approximate lengths \( I \) and \( J \) in the symmetry order curve and symmetry degree curve, respectively. Finally, we define the real number
\[
\ell = \frac{{\rm sup} \, I + {\rm sup} \, J}{2},
\]
which will be used to analyze the stability of the fullerene.

Figure \ref{best} displays the symmetry order curves and symmetry degree curves corresponding to fullerene  $\mathrm{C}_{20}, \mathrm{C}_{32}$, C$_{40}$ and $\mathrm{C}_{50}$, respectively.
The orange parts represent the symmetry order curves, while the green parts represent symmetry degree curves. By inspection, we obtain that
$$\ell_{20}\approx 2.216,~ \ell_{32}\approx 2.194, ~\ell_{40}\approx 2.175,  ~  \ell_{52}\approx 2.197.$$
The quantity $\ell$ we introduce exhibits a behavior similar to the average length of the bars representing the 1-dimensional persistent homology of fullerene molecules.
For example, Figure \ref{duibitu} presents the barcode representation of the persistent homology for fullerene $\mathrm{C}_{20}$ in dimension 1. It is seen that the $11$ equal-length bars vanish at around $2.34\mathrm{\AA}$, which is close to $\ell_{20}$. This is not a mere coincidence, but reflects an intrinsic relationship.
Both quantities, albeit from different perspectives, reflect to some extent the temporal dynamics of bond length changes in fullerene molecules.
	\begin{figure}[htbp]
		\centering
		\includegraphics[width=0.9\textwidth]{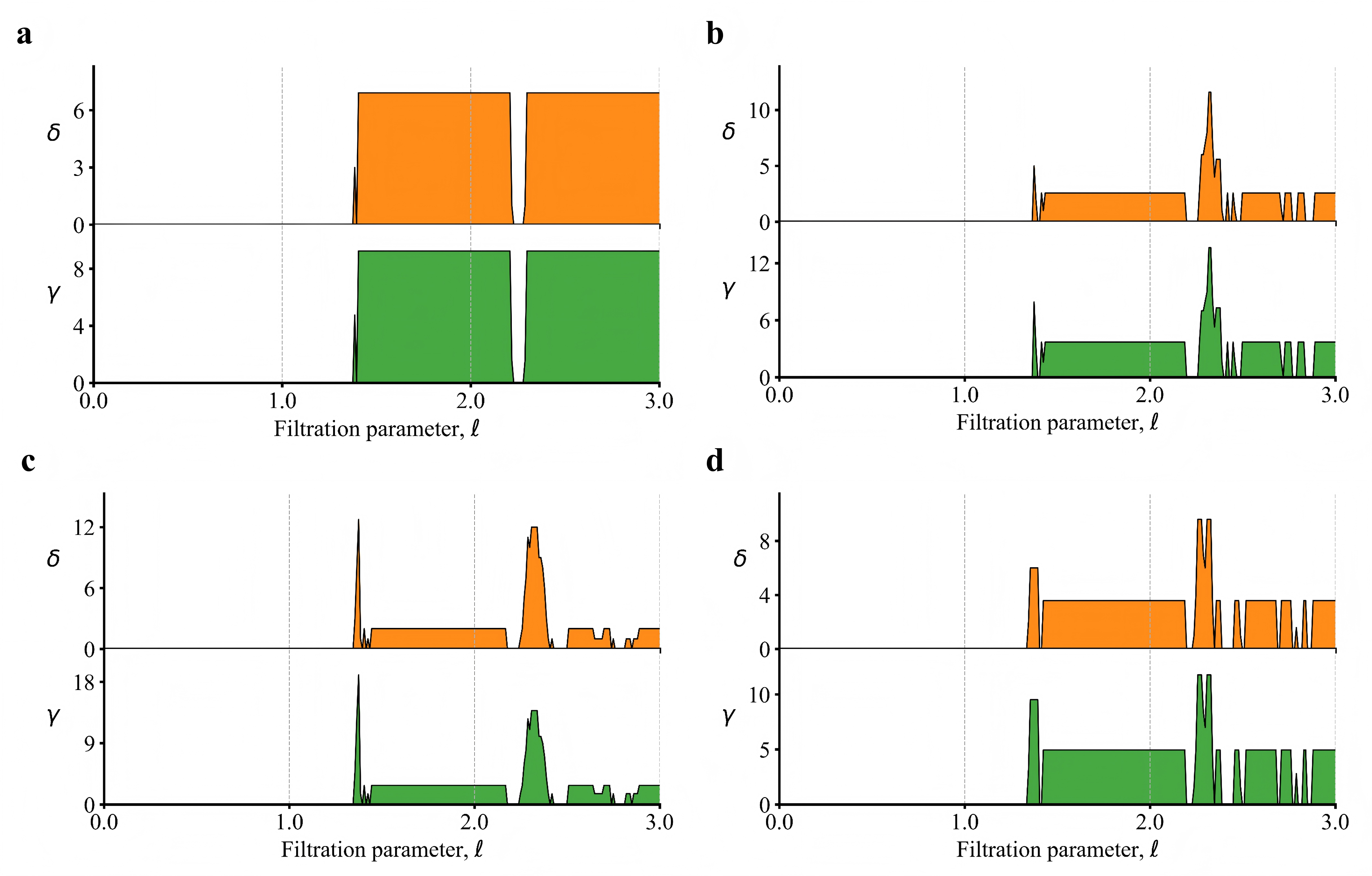}
		\caption{\textbf{a} The symmetry order curve and the symmetry degree curve for $\mathrm{C}_{20}$; \textbf{b} The symmetry order curve and the symmetry degree curve for $\mathrm{C}_{32}$; \textbf{c} The symmetry order curve and the symmetry degree curve for $\mathrm{C}_{40}$; \textbf{d} The symmetry order curve and the symmetry degree curve for $\mathrm{C}_{52}$.}
		\label{best}
	\end{figure}

In light of the above discussion,
in our model, we consider the quotient of $\ell$ and $n$, where $n$ is the number of carbon atoms. More precisely, by setting
$$R=\frac{\ell}{n},$$
we analyze the correlation between $R$ and the heat of formation energy.
In order to validate our predictions quantitatively, we apply the least squares approach to model the relationship between our predicted values and the heat of formation energy.
A correlation coefficient is given by
\begin{equation*}
C=\frac{\displaystyle \sum_{i=1}^{N}(R_i-\bar{R})(E_i-\bar{E})}{\left[\displaystyle \sum_{i=1}^{N}(R_i-\bar{R})^2\displaystyle \sum_{i=1}^{N}(E_i-\bar{E})^2\right]^{\frac{1}{2}}},
\end{equation*}
where $N = 12$ is the number of fullerene molecules  considered, $R_i$ represents the quotient of $\ell_i$ and the number of atoms, and $E_i$ is the heat of formation energy of the $i$-th fullerene molecule. The parameter $\bar{R}$ and $\bar{E}$
are the corresponding mean values.
	\begin{figure}[htbp]
		\centering
		\includegraphics[scale=0.6]{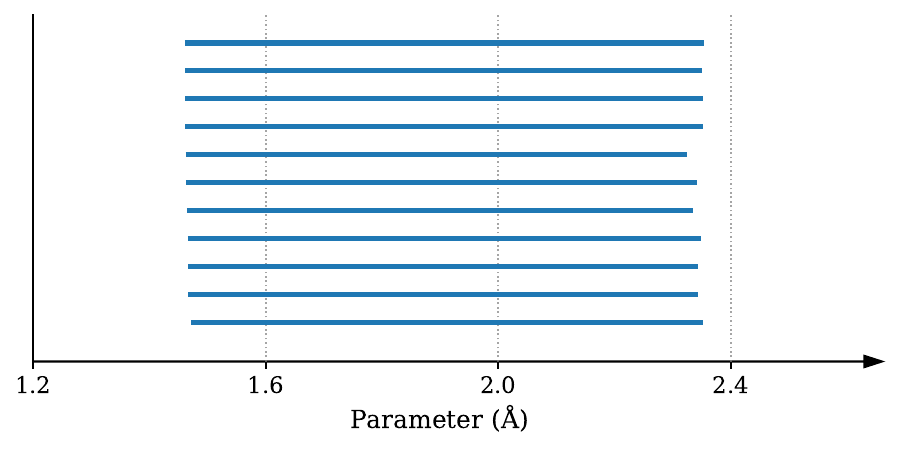}
		\caption{The barcode for the 1-dimensional homology of the fullerene molecule $\mathrm{C}_{20}$.}
		\label{duibitu}
	\end{figure}

Table \ref{table} lists the heat of formation energies  corresponding to different fullerene molecules. The unit for the heat of formation energy is eV/atom.
\begin{table}[h]
	\centering
	\begin{tabular}{c|c|c|c|c|c|c}
		\hline
		$N_{\text{atom}}$	& 20 & 24 & 26 & 28 & 30 & 32\\
		\hline
		$\text{Energy}$	& 1.180 & 1.050 & 0.989 & 0.912 & 0.850 & 0.781\\
		\hline\hline
		$N_{\text{atom}}$  & 36 &40 & 44 & 50 & 52 & 60\\
		\hline
		$\text{Energy}$	 & 0.706 &0.641 & 0.589 & 0.509 &0.502 &0.401 \\
		\hline
	\end{tabular}
	\caption{The heat of formation energy for small fullerene molecules.}
	\label{table}
\end{table}
The fitting result is depicted in Figure \ref{prediction}.
More specifically, Figure \ref{prediction} illustrates the comparison between our predicted values and the actual heat of formation energy. The left vertical axis corresponds to the heat of formation energy with the unit ev/atom, whereas the right vertical axis represents the ratio $R$ of the death time to the number of atoms. The red squares depict the variation of the heat of formation energy with respect to the number of atoms, while the green circles show the behavior of $R$ as a function of atomic number. The two lines align closely, confirming that our predictions capture the overall behavior of the heat of formation energy. The correlation coefficient reaches 0.979, suggesting the robustness of our model and underscoring the effectiveness of persistent automorphism module in making quantitative predictions. Compared to the correlation coefficients of 0.985 reported in \cite{KX2015} based on 9 fullerene molecules and 0.997 reported in \cite{JL2022} based on 9 fullerene molecules, our result achieves a correlation coefficient of 0.979 using 12 fullerene molecules, indicating greater generality.

Although our predictions for $\mathrm{C}_{20}$, $\mathrm{C}_{24}$ and $\mathrm{C}_{26}$ fullerenes deviate slightly from the reported energy profiles, which can be attributed to the data sets employed in our calculations are not identical to the ground-state data used in the literature, the essential features and relative trends of the energy profiles remain well captured. Furthermore, as illustrated in the figure, the two broken lines nearly coincide when the number of carbon atoms is greater than or equal to 30, which demonstrates the advantage of our model in predicting large fullerene molecules.
\begin{figure}[htbp]
	\centering
	\includegraphics[scale=0.45]{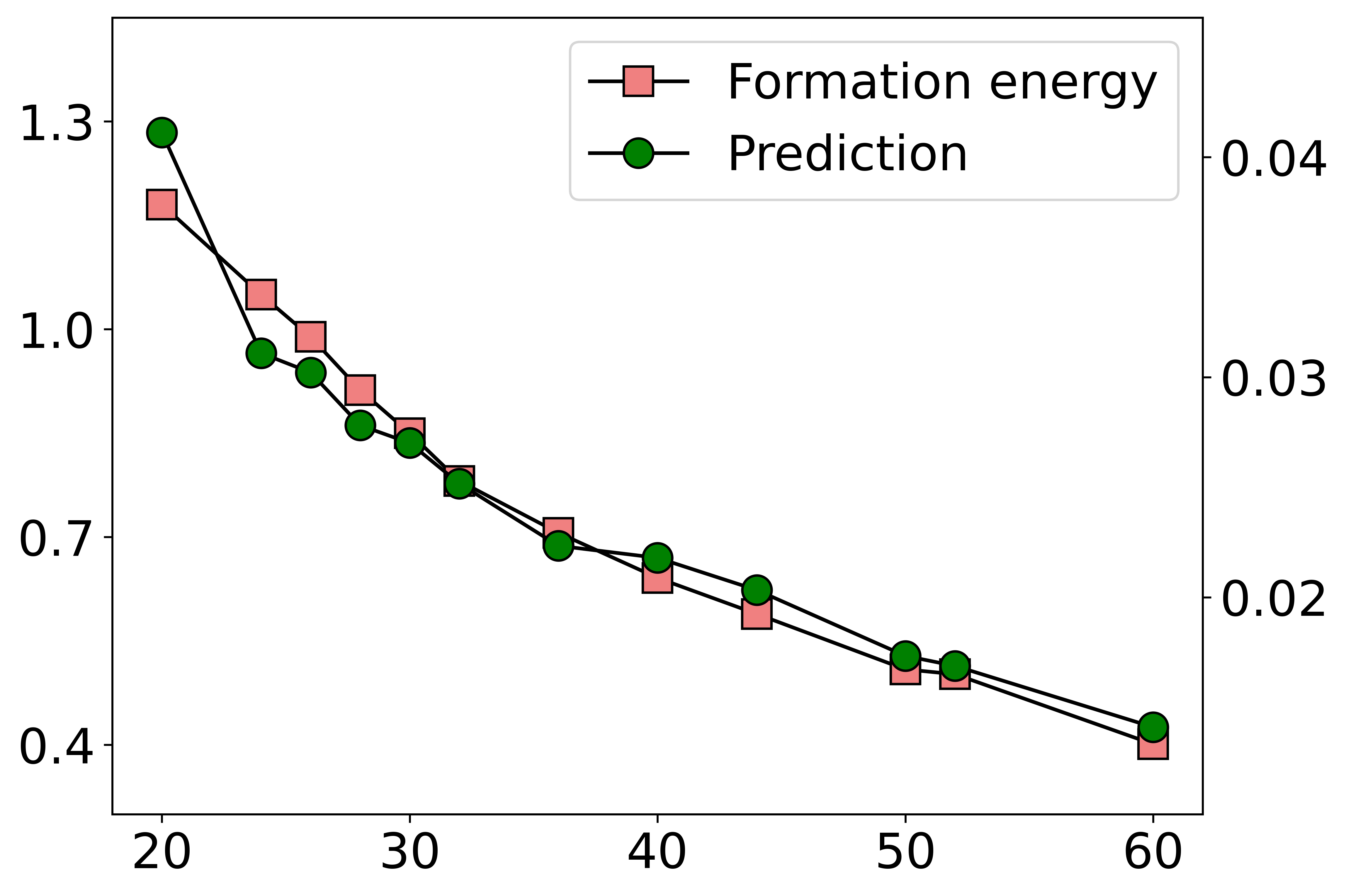}
	\caption{Comparison between the heat of formation energies and our model-predicted outcome.
	The left vertical axis corresponds to the heat of formation energy, expressed in eV per atom, whereas the right vertical axis corresponds to $R$.}
	\label{prediction}
\end{figure}

\section{Conclusion}

Topological methods, particularly persistent homology, have proven effective in capturing the inherent, stable structural features of complex data. This approach has demonstrated distinct advantages, especially in materials science and molecular biology. While topological features such as connected components, holes, and cavities reflect the structural information in data, the automorphism group of the data more effectively captures its symmetry information. Building on this insight, we integrate the multi-scale perspective of persistent homology into the automorphism group framework, thus introducing a multi-scale analysis of symmetries based on automorphism groups.

In this work, our main contribution is the introduction of multi-scale symmetry analysis (MSA). Specifically, our approach constructs the automorphism group of the Vietoris-Rips complex. For point cloud data, the Vietoris-Rips complex changes with respect to the distance parameter, leading to a multi-scale automorphism group. On one hand, we transform the computation of the automorphism group of simplicial complexes into the computation of the automorphism group of graphs, which simplifies the calculation of the Vietoris-Rips complex's automorphism group. On the other hand, since the automorphism group itself is not a functor, this presents a challenge in directly obtaining persistent symmetries. However, by modifying the graph category, we derive an automorphism module functor, ensuring the existence of persistent symmetry theory. In terms of application, we introduce the concepts of symmetry order and symmetry degree to characterize the richness of symmetry, and use these to analyze the multi-scale symmetries of fullerenes, ultimately predicting their stability. By predicting the stability of 12 fullerenes, we achieve a correlation coefficient of 0.979, which is more general and convincing compared to other models that predict fewer than 10 fullerenes.

In summary, multi-scale symmetry analysis holds great promise as a powerful tool for capturing and analyzing the inherent symmetries of data across different domains. With further theoretical development and application, it could open up new possibilities for understanding complex structures and patterns in materials science and molecular biology.

\section*{Acknowledgments}

This work was supported in part by the Natural Science Foundation of China (NSFC Grant No. 12401080), Scientific Research Foundation of Chongqing University of Technology and the Science and Technology Research Program of Chongqing Municipal Education Commission, grant number KJQN202501109.

\end{document}